\documentclass[12pt]{article}

\usepackage{amsmath,amsfonts,amssymb,amsthm}
\usepackage{amscd,color}
\usepackage{graphicx}

\newtheorem{theorem}{Theorem}[section]
\newtheorem{lemma}[theorem]{Lemma}
\newtheorem{prop}[theorem]{Proposition}

\newtheorem{remark}[theorem]{Remark}

\newtheorem{definition}[theorem]{Definition}

\newcommand{\R}{\mathbb{R}}
\newcommand{\C}{\mathbb{C}}

\newcommand{\Z}{\mathbb{Z}}

\newcommand{\K}{\mathbb{K}}

\newcommand{\Fcal}{\mathcal{F}}

\newcommand{\Jcal}{\mathcal{J}}
\newcommand{\Mcal}{\mathcal{M}}

\newcommand{\Scal}{\mathcal{S}}
\newcommand{\Gcal}{\mathcal{G}}

\newcommand{\Xcal}{\mathcal{X}}

\newcommand{\Ocal}{\mathcal{O}}

\newcommand{\gfrak}{\mathfrak{g}}

\newcommand{\mfrak}{\frak{m}}

\newcommand{\hfrak}{\frak{h}}
\newcommand{\gz}{{\frak{g}}^{0}}

\newcommand{\ufrak}{\frak{u}}
\newcommand{\nfrak}{\frak{n}}

\newcommand{\E}{\textup{\bf E}}

\begin{document}

\begin{titlepage}

\begin{centering}
{\LARGE \bf Spin Groups of Super Metrics and 

\vspace {.4cm}  a Theorem of Rogers}

\vspace{1cm}

\qquad \qquad {\bf  RONALD FULP} \\

\vspace {.5cm}

\qquad \qquad {\it Department of Mathematics} \\
\qquad \qquad {\it North Carolina State University } \\
\qquad \qquad {\it Raleigh NC 27695} \\
\qquad \qquad {\it fulp@math.ncsu.edu}

\vspace{.5cm}

December 2, 2014
\vspace{.5cm}
\begin{abstract} We derive the canonical forms of super Riemannian metrics and the local isometry groups of such metrics. For certain super metrics we also
compute the simply connected covering groups of the local isometry groups and interpret these as local spin groups
of the super metric. Super metrics define reductions $\Ocal \Scal_g$ of the relevant  frame bundle. When principal bundles $\widetilde \Scal_g$ exist with structure group the simply
connected covering group $\tilde \Gcal$ of the structure group of $\Ocal \Scal_g,$ 
representations of $\tilde\Gcal$ define vector bundles associated to $\widetilde \Scal_g$ whose sections are "spinor fields" associated with the super metric $g.$
Using a generalization of a Theorem of Rogers, which is itself one of the main results of this paper, we show that for super metrics we call body reducible, each such simply connected covering group $\tilde \Gcal$ is a super Lie group
with a conventional super Lie algebra as its corresponding super Lie algebra.

Some of our results were known to DeWitt \cite{Dewitt} using formal Grassmann series and others were known by Rogers 
using finitely many Grassmann generators and passing to a direct limit. We work exclusively in the category of $G^{\infty}$
supermanifolds with $G^{\infty}$ mappings. Our supernumbers are infinite series of products of Grassmann generators
subject to convergence in the $\ell_1$ norm introduced by Rogers \cite{Rsuperman,R}.\newline


 \end{abstract}

\end{centering}

\end{titlepage}

\pagebreak

\section{Introduction} The canonical form of a super Riemannian metric $g$ has been known \cite{Dewitt} for some time
in the case that the supermanifold on which $g$ is defined is modeled on super  vector spaces of the form $\R^{m|n}$ where
the components of the vectors in $\R^{m|n}$ are supernumbers defined as formal power series of products of Grassmann generators.
In this case there are no convergence issues. There seems to be no corresponding result when the underlying space of supernumbers are series of products of Grassmann generators defined by requiring that the series converge in the $\ell_1$ norm defined by Rogers
\cite{Rsuperman,R}. It is our intent to determine such canonical forms in this case and to compute the corresponding local isometry groups $\Gcal$ along with their simply connected covering groups $\tilde \Gcal.$ To do this we require a generalization of a Theorem of Rogers. It turns out that using this generalization we can show that for a large class of super metrics, called body reducible super metrics,  the simply connected covering group of each local isometry group is a super Lie group whose super Lie algebra is a conventional super Lie algebra. Moreover, each such covering group is a semi-direct product  of an ordinary simply connected finite-dimensional Lie group  with a group $N$ which is an infinite-dimensional generalization of a nilpotent group called a quasi-nilpotent group by Wojtyn'ski'\cite{W} who discovered groups of this type.  When representations of  factors of the semi-direct product  $\tilde \Gcal$ are known each such representation  gives rise to a representation of the semi-direct product  itself.  In every case representations of $\tilde \Gcal$  define induced vector bundles and we think of the local sections of such vector bundles as  local spinor fields associated
with the super metric. In a real sense these fields have a right to be called super spinor fields but the notion of super spin also arises in quantum field theory \cite{Q} where super spin is a quantum number obtained from  supersymmetric extensions of the Poincare algebra. Quantum numbers of this type have already proven to be useful in that context. We do not know of applications of the type of super spinor field discussed in our work here but it would seem that there should be applications to supergravity.

In this paper we also prove a generalization of a Theorem of Rogers. We believe this result has significance in its own right quite aside from its present application to the program outlined above. We show that if $\gfrak=\gz \oplus \gfrak^1$ is a $\Z_2$ graded Banach 
Lie algebra and $\hfrak=(\Lambda \otimes \gfrak)^0$ is the even part of corresponding conventional super Lie algebra $\Lambda \otimes \gfrak$ then the kernel of
the body mapping from $\hfrak$ into $\gz$ is a quasi-nilpotent Banach Lie algebra in the sense of Wojtyn'ski \cite{W}. This Lie algebra
has a global Banach Lie group structure defined via the exponential mapping and so is a quasi-nilpotent Banach Lie group. If we denote this group by $N$ and if  $\tilde G$ is the unique simply connected 
finite-dimensional Banach Lie group having $\gz$ as its Lie algebra, then we prove that there is an action of $\tilde G$ on $N$
which defines a semi-direct product structure  $H=\widetilde G \times N.$ Moreover there is a $G^{\infty}$ atlas on $H$ relative to which $H$ is a $G^{\infty}$ super Lie group and this super Lie group has $\hfrak$ as its corresponding super Lie algebra $\hfrak.$

Rogers \cite{R} proves a version of this Theorem in her book. She shows that in the case that the supernumbers are generated by a finite number of Grassmann generators then this result holds. In that case one has at ones disposal  results from the theory of finite-dimensional Lie groups and Lie algebras. She is able to recover an infinite-dimensional version of her theorem by passing to a direct limit. She therefore has an analytic version of her theorem which utilizes the direct limit topology. 

In our version of the Theorem we stick with $G^{\infty}$ structures rather than analytic structures and we use Rogers $\ell_1$ 
topology at various stages of our proof. Our proof follows hers but differs in details since it necessarily requires the appropriate modifications in order  to remain 
within the category of infinite-dimensional Banach Lie groups and Banach Lie algebras.

In Section 2 we lay out our preliminary definitions and conventions. We follow the conventions of Rogers \cite{R} in the main,
but at times use slightly different notation. In subsections 1 and 2 of Section 3 we define the different norms required to make 
contact with the results of Wojtyn'ski's \cite{W} . In subsection 3 of Section 3 we prove our version of Rogers theorem.
In the fourth and final section of the paper we present our results on super Riemannian metrics.

\section{Preliminaries}

In this section we briefly introduce our conventions and assumptions regarding supernumbers, superspace, supermanifolds and basic constructions.

The starting point is a Grassmann algebra with coefficients in the field $\K$ which denotes either the 
field of real numbers or the field of complex numbers.We assume the Grassmann algebra is countably infinitely generated with generators ${\zeta}^{1}, {\zeta}^{2},{\zeta}^{3} \cdots $ which are  anticommuting indeterminates;
\begin{equation}
{\zeta}^{i}{\zeta}^{j}= -{\zeta}^{j}{\zeta}^{i}.
\end{equation}
As usual if $i=j$ the square of any Grassmann generator is zero.

Although we sometimes are not explicit regarding our convention we utilize a multi-index notation $I=(i_{1},i_{2},\cdots, i_{k})$ with $1 \leq i_{1}< i_{2}< \cdots <i_{k} $
and denote the set of all increasing strings of positive integers of length $k$  by $\mathcal{I}_{k}.$
 We write

\begin{equation}
z=\sum_{p=0}^{\infty}\sum_{I \in \mathcal{I}_{p}}^{}z_{I}{\zeta}^{I}
\end{equation}

\noindent where for $I \in \mathcal{I}_{k},$
\begin{equation}
{\zeta}^{I} = {\zeta}^{(i_{1},i_{2},\dots , i_{k})} = {\zeta}^{i_{1}}{\zeta}^{i_{2}}\dots {\zeta}^{i_{k}}.
\end{equation}

\begin{equation}
\begin{array}{l}
z_{I} = z_{i_{1},i_{2},\dots i_{k}}\in \K,  \ \ \ I \in \mathcal{I}_k, \ \ \ k \geq 1 \\ 
\end{array}
\end{equation}
\noindent and
\begin{equation}
\begin{array}{l}

z_{I} = z_{0} \in \K, \ \ \ I \in \mathcal{I}_{0}.
\end{array}
\end{equation}

Notice that, for convenience, wherever the null-index  $I \in \mathcal{I}_{0}$ appears as an index on an element it is denoted more simply by the label "0".

Now the set $\Lambda$ of all supernumbers is, by definition, the set of all formal power series
\begin{equation}
z=\sum_{p=0}^{\infty}\sum_{I \in \mathcal{I}_{p}}^{}z_{I}{\zeta}^{I}
\end{equation}
such that
$$||z||= \sum_{p=0}^{\infty}\sum_{I \in \mathcal{I}_{p}}^{}|z_{I}|$$
converges. Since the set $\mathcal{I}=\cup_{k=0}^{\infty}\mathcal{I}_{k}$ is countable the space $\Lambda$ can be identified, as a Banach space,  with the Banach space $\ell_1$ of  summable sequences of real or complex numbers. As is indicated by the notation $||\cdot ||$ is the norm on the space. It is well-known  \cite{Rsuperman} that  $\Lambda$ is in fact a Banach algebra and in this sense it has more structure than  $\ell_1.$

$\Lambda^0$ denotes the closed subspace of even elements $z=\sum_{|J|=even}z_{J}{\zeta}^{J}$ and $\Lambda^1$ denotes the closed subspace of odd elements $z=\sum_{|K|=odd }z_{K}{\zeta}^{K}$ of $\Lambda.$ As usual $\K^{p|q}$ denotes the Banach space $(\Lambda^0)^p \times (\Lambda^1)^q$ relative to the norm defined by
$$||(x^1,x^2,\cdots,x^p,\theta^1,\theta^2,\cdots, \theta^q)||=\sum_{i=1}^p||x^i||+\sum_{\alpha=1}^q||\theta^{\alpha}||.$$
We  follow the notation and conventions of Rogers \cite{R} regarding basic terminology. In particular superconjugation is defined as in Rogers \cite{R}. A supernumber $z$  over the field $\C$ is real if and only if $z^*=z$ and this is the case iff all the coefficients $z_I$ are in fact real. The supervector space $\R^{p|q }$ consists of supervectors in $(\Lambda^0)^p \times (\Lambda^1)^q$ whose components are real while
$\C^{p|q }$ consists of supervectors in $(\Lambda^0)^p \times (\Lambda^1)^q$ whose components are complex. 

All supermanifolds will be defined by atlases whose charts take their values in Banach spaces which are coordinated supervector spaces modeled on $\R^{p|q},$ for some $p,q.$\\

An exceptionally good paper dealing with the foundations of superanalysis is \cite{JP}. In this paper Jadczyk and Pilch show that a mapping is supersmooth (of class $G^{\infty}$)  if and only if it is superdifferentiable (of class $G^1$) and is smooth (of class $C^{\infty}).$ We utilize this equivalence throughout the paper often without specifically mentioning the fact. We refer to Rogers fundamental papers \cite{Rsuperman, RLie}  as well as her book \cite{R} for most conventions.

Finally, throughout Section 2 we allow $\K$ to be either the field of  real numbers or the field of complex
numbers, but in section 3 we restrict our attention to the real field only. Our super metric is super real number valued.

\section{Rogers Theorem Part I}
In this section we show how to generalize a Theorem of Rogers \cite{R} which shows how to construct
a simply connected super Lie group from a conventional super Lie algebra. Since our notation is slightly different from Rogers and since our super Lie algebras are $\Lambda$ modules where 
$\Lambda$ is the infinite-dimensional Banach algebra defined in the introduction we spell out our conventions in the next few paragraphs. 

\subsection{Conventional super Lie algebras with norm}
Let $\gfrak=\gz \oplus \gfrak^1$ denote a $\Z_2$ graded Banach Lie algebra over $\K$ of dimension $(m,n),$ i.e., we require that $m=dim_{\K} \  \gz,$ that $n=dim_{\K} \ \gfrak^1,$ and that there is a norm $|| \cdot ||_{\gfrak} $ on $\gfrak$ relative to which $\gfrak$ is a Banach space with the property that
$|| [X,Y] ||_{\gfrak} \leq || X ||_{\gfrak} || Y ||_{\gfrak}$ for all $X,Y\in \gfrak.$  We follow the conventions of Buchbinder and Kuzenko \cite{BB} for whom a super Lie algebra $\mfrak$ is a $\Z_2$ graded Lie algebra which is also a $\Lambda$ module. Such a super Lie algebra is "conventional" if there exists a $\Z_2$ graded Lie
algebra $\gfrak$ such that $\mfrak=\Lambda\otimes \gfrak.$

We wish to show that if $\ufrak=\Lambda \otimes \gfrak $ is a conventional super Lie algebra which is appropriately normed
then there is a simply connected super Lie group having Lie algebra $\ufrak.$ Rogers \cite{R} proves this theorem in the case that $\Lambda$ is a finitely generated Grassman algebra. We prove her result when $\Lambda$ has (countably) infinitely many generators. In this case $\Lambda$ is still a Banach Lie algebra, a fact which enables us to induce a norm on $\ufrak$ relative to which 
$\ufrak$ becomes a Banach super Lie algebra.

For simplicity of exposition we choose a basis $\{ X_i \}_{i=1}^{m+n}$ of $\gfrak$ once for all such that $\{ X_i \}_{i=1}^{m}$ is a basis of $\gz$ and $\{ X_i \}_{i=m+1}^{m+n}$ is a basis of $\gfrak^1.$
We define a norm on $\ufrak$ utilizing this basis but in the end we only need a norm having certain properties and the norm we define using this fixed basis serves to show that such norms exist having the desired properties.

Since $\ufrak$ is freely and finitely generated as a $\Lambda$ module we see that if $Y\in \ufrak,$  then $Y=\sum_{i=1}^{m+n} (y^i \otimes X_i) \equiv \sum_{i=1}^{m+n} y^i X_i$ for a unique sequence  $\{ y_i\}$ in $\Lambda.$ We define
$$|| Y||_{\Lambda \otimes \gfrak}=\sum_{i=1}^{m+n} ||y^i ||_{\Lambda}||X_i||_{\gfrak}.$$
It is easy to show that $|| \ ||_{\Lambda \otimes \gfrak}$ is indeed a norm on $\ufrak$ and that the space $\ufrak$ is complete with respect to this norm. Notice also that since $\Lambda$ is a Banach algebra, it follows that $||\lambda Y||_{\ufrak} \leq ||\lambda||_{\Lambda} ||Y ||_{\ufrak}$ for $\lambda \in \Lambda, Y\in \ufrak.$
Finally,  if $Y,Z\in \ufrak, Y=\sum_iy^iX_i,Z=\sum_jz^jX_j, $ then 
$$|| [Y,Z] ||_{\ufrak}=|| \sum_{i,j} y^iz^j(-1)^{ij}[X_i,X_j]||_{\ufrak}\leq \sum_{ij} ||(y^i z^j) [X_i,X_j] \  ||_{\ufrak} $$
 $$ \leq \sum_{ij} ||(y^i z^j) ||_{\Lambda}  \ || [X_i,X_j] ||_{\gfrak} \leq \sum_{ij} ||y^i ||_{\Lambda}
  \  ||z^j ||_{\Lambda} \  ||X_i ||_{\gfrak} \ ||X_j ||_{\gfrak}=||Y||_{\ufrak}  \ ||Z||_{\ufrak}.$$

Thus $\ufrak$ is a Banach Lie algebra. 
From this point on we drop the subscripts on the various norms since it should be clear from the context which norm is which.

Observe that the elements $Y$ of the Lie algebra $\ufrak$ may also be expressed in terms of
a Banach basis over $\K.$ One has that $y=\sum_iy^i X_i$ for some unique sequence  $\{y^i\}$ in $\Lambda$ but $\Lambda$ itself has a Banach basis in the sense that every element  
$y\in \Lambda$ has a series expansion $y=\sum_J y_J \zeta^J$ where the $\{\zeta^j \}$ are the generators of $\Lambda $ and where we have used multi-index notation: $J=(j_1,j_2,\cdots,j_p).$
Thus $Y\in \ufrak$ may be written as 
$$Y=\sum_i(\sum_J y^i_J \zeta^J)X_i=\sum_i\sum_J y^i_J (\zeta^J\otimes X_i)=
\sum_i\sum_J y^i_J  X^J_i$$
\noindent where $X^J_i$ is defined to be  $(\zeta^J \otimes X_i)\in \ufrak$ for each $i,J.$
Since $\zeta^{\emptyset}=1,  X^{\emptyset}_i=X_i$ for each $i.$
It follows that the set $\{X^J_i\}$ of elements of $\ufrak$ is a Banach basis of $\ufrak$ over $\K.$

It turns out to be useful to extend the body and soul mappings to apply to elements of $\ufrak.$
If $X=\sum_i\lambda^i X_i \in \ufrak,$ define $\beta(X)=\sum_i \beta(\lambda^i)X_i$ and $s(X)=\sum_is(\lambda^i)X_i.$ As usual, $X=\beta(X)+s(X).$ Notice that these definitions are independent of the
basis used to define them since any two bases $\{X_i\}$ and $\{Y_j\}$  of $\gfrak$ are related by 
$Y_i=\sum_jk_i^jX_j$ where the $k^j_i $ are ordinary numbers. Consequently 
$$\beta(\sum_i\mu^iY_i)=\beta(\sum_j(\sum_i k_i^j\mu^i)X_j)=\sum_j \sum_i\beta( k_i^j  \mu^i)X_j=
\sum_i\beta(\mu^i)Y_i.$$

\noindent Clearly, for $X,Y\in \ufrak, \beta(X+Y)=\beta(X)+\beta(Y)$ and if $\lambda\in \Lambda,$ then $\beta(\lambda X)=\beta(\lambda)\beta(X).$ While $\beta$ is not generally a Lie homomorphism 
its kernel as a linear mapping is a Lie ideal in $\ufrak.$ Indeed, if  $X=\sum_i\lambda^i X_i\in \ufrak$ and $Y=\sum_j\mu^j X_j\in ker(\beta),$ then $\beta(\mu^j)=0$ for all $j$ and 
$$\beta[X,Y]=\beta(\sum_{i,j}(\lambda^i \mu^j)(-1)^{ij}f_{ij}^kX_k)=
\sum_{i,j}(\beta(\lambda^i)(\beta( \mu^j))(-1)^{ij}f_{ij}^kX_k=0$$
where the $f^k_{ij}$ are the structure constants of the graded Lie algebra. Consequently, $ker(\beta)$
is an ideal and also a sub Lie algebra of $\ufrak.$ Notice that $||\beta(X)|| \leq ||X||$ and consequently
$\beta$ is continuous.  Define, once for all, $\nfrak$ to be the kernel of the mapping $\beta:\ufrak \rightarrow \ufrak.$ Since $\beta$ is continuous, $\nfrak$ is closed as a subspace of the Banach space 
$\ufrak$ and so is a closed ideal of $\ufrak.$

In our proof below we utilize the even subalgebra $\hfrak=
  (\Lambda^0 \otimes \gz) \oplus (\Lambda^1 \otimes \gfrak^1) $ of $\ufrak.$
Notice that when elements of $\hfrak$ are expanded in the basis over $\K$ one must restrict the multi-indices 
$J$ such that $|J|$ is even when $X_i$ is even and such that $|J|$ is odd when $X_i$ is odd.
Also notice that $|J|$ may be zero and elements of the body of $\hfrak$ may be written in the 
form $\sum_i\sum_{|J|=0} y^i_J  X^J_i=\sum_i y_{\emptyset}^i X_i.$

These facts enable us to utilize a portion of Rogers' proof almost word for word in our proof below.
In her proof the basis $\{X^J_i\}$ over $\K$ is finite and so is a basis in the usual sense for finite dimensional Lie algebras. In our case the basis $\{X^J_i\}$ is infinite but is a basis in the sense
of Banach. Our series converge absolutely and consequently we are able to use a part of her
argument without change.

\bigskip

Our argument, which follows that of Rogers, precedes as follows. Let $\widetilde G$ denote that unique 
finite dimensional simply connected Lie group having $\gz$ as Lie algebra. Now consider the even part 
$\hfrak=(\Lambda^0 \otimes \gz) \oplus (\Lambda^1 \otimes \gfrak^1)$ of $\ufrak.$ Split $\hfrak$
into two parts
 $$\hfrak= \beta( \hfrak) \oplus s( \hfrak)=
 \gz \oplus [(s (\Lambda^0)\otimes \gz )\oplus (\Lambda^1 \otimes \gfrak^1)]$$
  
\noindent and observe that $(s (\Lambda^0)\otimes \gz )\oplus (\Lambda^1 \otimes \gfrak^1)$ is 
precisely the kernel $\nfrak$  of the body mapping from $\hfrak$ onto $\gz$ and consequently is a closed Lie ideal in both $\hfrak$ and $\ufrak.$

Now $\widetilde G$ is the Lie group of $\gz$ and one seeks to find the Lie group of $\nfrak = [(s (\Lambda^0)\otimes \gz) \oplus (\Lambda^1 \otimes \gfrak^1)].$
In the finite dimensional case $\nfrak$ is nilpotent and its simply connected Lie group $N$  is a nilpotent Lie group. Rogers shows that there is an action of $\widetilde G$ on $N$ and that one can construct a semi-direct product $H$ of  $\widetilde G$ and $N$ which is simply connected, which has the structure of a super Lie group, and whose Lie algebra is $\hfrak.$ In our case the kernel $\nfrak$ fails to be nilpotent but we will show below that it is quasi-nilpotent in the sense defined by Wojtyn'ski \cite{W}. It turns out that there is  a simply connected infinite dimensional Banach Lie group $N$ with Lie algebra $\nfrak,$ and, in fact, $N$ and $\nfrak$ are the same as Banach spaces. We will show that the action of Rogers
is valid in our case and that we can construct a Banach Lie group $H$ which has all the required
properties.

In order to apply Wojtyn'ski's \cite{W} results we need a couple of Lemmas.

\begin{lemma} Assume that $X\in \nfrak$ and that $M$ is the matrix of  $ad_X:\ufrak \rightarrow \ufrak$
relative to a basis $\{X_i\}$ of $\gfrak.$ Then $\beta(M)=0.$
\end{lemma}

\begin{proof} Write $X=\sum_i\lambda^iX_i, $ then $0=\beta(X)=\sum_i\beta(\lambda^i)X_i$ and
consequently, $\beta(\lambda^i)=0$ for all $i.$ Now
$$\beta(ad_X(X_i))=\beta[X,X_i]=\beta(\sum_j \lambda^j[X_j,X_i])=\sum_j\sum_k\beta(\lambda^jf_{ji}^k)X_k=0$$
where the  $\{f_{ij}^k\}$ are the structure constants of $\gfrak.$ We also have $0=\beta(ad_X(X_i))=
\beta(\sum_j M_i^jX_j)=\sum \beta(M_i^j)X_j.$ Consequently $\beta(M)=(\beta(M_i^j))=0.$
\end{proof}

Notice that for $Y\in \nfrak,$ $Y=\sum_{a=1}^m x^a X_a+\sum_{\alpha=1}^{n}\mu^{\alpha}X_{m+\alpha}$ where the supernumbers $x^a$ are even, the supernumbers $\mu^{\alpha}$ are odd, and $ \beta(x^a)=0.$
Moreover $Y$ is uniquely determined by the vector $(x^1,x^2,\cdots,x^m,\mu^1,\cdots,\mu^n)\in \widetilde \K^{m|n}$ where $\widetilde \K^{m|n}$ is the set of vectors in $ \K^{m|n}$ with zero bodies.
In the next lemma we denote the vector 
$(x^1,x^2,\cdots,x^m,\mu^1,\cdots,\mu^n)\in \widetilde \K^{m|n}$ which determines $Y\in \nfrak $ by $\vec Y.$

In the next Lemma we abuse notation by writing $M^a_{\alpha} , M^{\beta}_{\alpha}$ instead of
$M^a_{\alpha +m} , M^{\beta+m}_{\alpha +m}.$

\begin{lemma}  \label{L:spectrum} Let $X\in \nfrak$ and let $M$ denote the matrix of $ad_X: \ufrak \rightarrow \ufrak $
relative to the basis $\{X_i\}$ of $\gfrak.$ The matrix $M$ also represents the linear mapping 
$ad_X$ as a mapping from $\nfrak$ into itself in the sense that if $Y\in \nfrak$ and 
$\vec Y=(x^1,x^2,\cdots,x^m,\mu^1,\cdots,\mu^n)$ then $ Z=ad_X(Y)$ is represented by 
$\vec Z=(y^1,y^2,\cdots,y^m,\lambda^1,\cdots, \lambda^n)$ where
$$y^b=\sum_ax^aM_a^b+\sum_{\alpha}(-1)^{\alpha  X} \mu^{\alpha} M_{\alpha}^b$$ 
$$\lambda^{\beta}=\sum_ax^aM_a^{\beta}+\sum_{\alpha}(-1)^{\alpha  X} \mu^{\alpha} M_{\alpha}^{\beta}.$$ 
Define a matrix $M_{\nfrak}$ by

     \[M_{\nfrak}= \left  ( \begin{array}{ccc}
                                           M^b_a     &   (-1)^{\alpha X}M^b_{\alpha} \\
                                                  
                                          M_a^{\beta}    &    (-1)^{\alpha X}M_{\alpha}^{\beta}
                                         
                                          \end{array} \right ).   \]

\noindent then it follows that $Z=ad_X(Y)$ is equivalent to the matrix equation
$\vec Z^T =M_{\nfrak}\vec Y^T.$     Consequently,  for each $\xi\in \K, \xi I_{\nfrak} - ad_X $
is invertible iff $\xi I_{\widetilde \K^{m|n}} - M_{\nfrak} $ is invertible. Moreover the latter
is invertible iff $\xi \neq 0.$                                  
\end{lemma}

\begin{proof} The two equations in the statement of the lemma follow immediately by writing 
$$Y=\sum_{a=1}^m x^a X_a+\sum_{\alpha=1}^{n}\mu^{\alpha}X_{m+\alpha}, $$
$$ ad_X(X_b)=\sum_{a=1}^m M_b^aX_a+\sum_{\alpha=1}^nM_b^{\alpha}X_{\alpha+m}, $$
$$ad_X(X_{\beta+m})=\sum_{a=1}^m M_{\beta}^aX_a+\sum_{\alpha=1}^nM_{\beta}^{\alpha}X_{m+\alpha},$$ and then by computing $Z=ad_X(Y).$ The matrix formula $\vec Z^T =M_{\nfrak}\vec Y^T$
is an immediate consequence of the two equations. The invertibility of $\xi  I_{\nfrak} -ad_X$ is clearly 
equivalent to invertibility of its matrix  $\xi  I_{\widetilde \K^{m|n}} - M_{\nfrak} .$ The body of the matrix
$\xi  I_{\widetilde \K^{m|n}} - M_{\nfrak} $ is $\xi  I_{\widetilde \K^{m|n}}$  since, by the preceding lemma,
the body of $M$ is zero and therefore so is the body of $M_{\nfrak}.$ But a matrix is invertible iff its body is invertible. Consequently, $\xi  I_{\widetilde \K^{m|n}} - M_{\nfrak} $ is invertible iff $\xi \neq 0.$
\end{proof}

\subsection{Baker-Campbell-Hausdorff and quasi-nilpotent Lie algebras} 

According to Wojtyn'ski \cite{W} it is known that for a Banach Lie group $\widetilde G$ the group multiplication in exponential coordinates is locally described by the Baker-Campbell-Hausdorff series. 
Stating this result in more detail, let $\gfrak$ denote the Lie algebra of $\widetilde G$ and  define $\Theta = W\circ Z$
where $W$ and $Z$ are the formal series defined by

$$W(z)=log(1+z)=\sum_{n=1}^{\infty}(-1)^{n+1}\frac{z^n}{n}$$ 
and 
$$Z(x,y)=e^xe^y-1=\sum_{j+k\geq 1}\frac{x^j}{j!}\frac{y^k}{k!}.$$

\noindent Then $\Theta $ is, in fact, the B-C-H series,
$$\Theta(x,y)=\sum_{m=1}^{\infty} \Theta_m(x,y)$$
where $\Theta_m(x,y)$ is the finite sum of all homogeneous terms of order $m.$
If one defines an operation $\diamond$ on an appropriate open subset $U$ of $\gfrak$ by $X\diamond Y= \Theta (X,Y),$ for $ X,Y \in U,$  then $U$ is a local  Banach Lie group relative to this
operation. Moreover if $U$ is chosen appropriately then the exponential maps $U$ bijectively
onto an open subset of $\widetilde G$ containing the identity of $\widetilde G$ and
$$exp(X)exp(Y)=exp(X \diamond Y).$$
Thus the exponential is a local Lie group isomorphism.
It is pointed out by  Wojtyn'ski \cite{W}  that the following proposition is well-known:

\begin{prop}\label{P:LocalLie} If $\gfrak$ is a Banach Lie algebra which is normed in such a way that
$||[a,b]||\leq ||a|| ||b||,$ for all $a,b \in \gfrak$ then the B-C-H series converges for all pairs $(a,b)\in \gfrak \times \gfrak$ such that 
$||a|| +||b|| \leq ln(2).$ It follows that for t sufficiently small,
$$exp(ta)exp(tb)=exp(t(a+b)+\frac{1}{2} t^2 [a,b]+ O(t^3)).$$

\end{prop}

\bigskip

Wojtyn'ski \cite{W} finds necessary and sufficient conditions that the B-C-H series converges  globally
and in fact shows that this is the case if and only if the Banach Lie algebra is quasi-nilpotent, i.e.,
if and only if the endomorphism $ad_a:\gfrak \rightarrow \gfrak$ has zero spectrum for each $a\in \gfrak.$
In this case the operation $\diamond$ is globally defined and $\gfrak$ is a Banach Lie group relative 
to this operation. It is the unique simply connected Lie group with Lie algebra $\gfrak.$ In this  case
the exponential is identified with the identity mapping from the Lie algebra $\gfrak$ onto the Lie group $\widetilde G=\gfrak.$

\bigskip
\bigskip
\begin{lemma} Assume that $M$ is a manifold modeled on $\K^m,$ that $U$ is an open
subset of $M$ and that $\phi$ is a $C^{\infty} $ chart of $M$ defined on $U.$ Let $N$ be a manifold modeled on the Banach space  $\widetilde{\K^{m|n}}$ with a globally defined $C^{\infty}$ chart $\eta$ which we identify with the identity mapping (thus we identify $N$ with $\widetilde{\K^{m|n}}$).
If we identify   $\K^{m|n}$ with $\K^m \times \widetilde{ \K^{m|n}}$ via the mapping 
$$(z^1,\cdots, z^{m+n}) \mapsto 
((\beta(z^1),\cdots,\beta(z^m)),(s(z^1),\cdots ,s(z^m),z^{m+1}, \cdots,z^{m+n})),$$
where $\beta$ is the body mapping and $s$ is the soul mapping,
then the differential $d\psi$ of the mapping $\psi$ from $U\times N$ onto 
$\phi(U)\times \widetilde{ \K^{m|n}}$ defined by $\psi=\phi \times \eta$ is a $\Lambda^0$ right-linear mapping from $\K^{m|n}$ onto itself.
\end{lemma}

\begin{proof} If $a\in \Lambda^0,$ then under the identification defined in the statement
of the lemma $a(z^1,\cdots, z^{m+n})=(az^1,\cdots, az^{m+n})$ is identified with
$$((\beta(az^1),\cdots,\beta(az^m)),(s(az^1),\cdots ,s(az^m),az^{m+1}, \cdots,az^{m+n})).$$
So the action on the factor $\K^m$ becomes $$a((\beta(z^1),\cdots,\beta(z^m))=\beta(a)(\beta(z^1),\cdots, \beta(z^{m}))$$
while on the factor  $\widetilde{ \K^{m|n}}$
$$a(s(z^1),\cdots ,s(z^m),z^{m+1}, \cdots,z^{m+n}))=(s(az^1),\cdots ,s(az^m),az^{m+1}, \cdots,az^{m+n})).$$
Now $d\psi=d\phi \times d\eta=\phi \times Id$ acts on $a(z^1,\cdots, z^{m+n})$ to give
$$(d\phi(\beta(a)(\beta(z^1),\cdots,\beta(z^m)),d\eta(s(az^1),\cdots ,s(az^m),az^{m+1}, \cdots,az^{m+n}))=$$
$$(\beta(a)d\phi((\beta(z^1),\cdots,\beta(z^m)),(s(az^1),\cdots ,s(az^m),az^{m+1}, \cdots,az^{m+n}))=$$
$$(\beta(a)d\phi((\beta(z^1),\cdots,\beta(z^m)),a(s(z^1),\cdots ,s(z^m),z^{m+1}, \cdots,z^{m+n}))=$$
$$a(d\phi((\beta(z^1),\cdots,\beta(z^m)),(s(z^1),\cdots ,s(z^m),z^{m+1}, \cdots,z^{m+n}))=$$
$$a(d\phi \times d\eta)((\beta(z^1),\cdots,\beta(z^m)),(s(z^1),\cdots ,s(z^m),z^{m+1}, \cdots,z^{m+n}),$$
\noindent which is precisely $a \ d\psi(z^1,z^2,\cdots,z^{m+n}).$
It follows that $d\psi(a(z^1,\cdots, z^{m+n})=a \ d\psi(z^1,z^2,\cdots,z^{m+n})$ as required.
\end{proof}

\subsection{Rogers' Theorem Part  II}

Recall from Rogers' Theorem Part I, that  $\gfrak=\gz \oplus \gfrak^1$ denotes a $\Z_2$ graded Banach Lie algebra over $\K$ of dimension $(m,n),$ that $\ufrak=\Lambda\otimes \gfrak,$ and that $\hfrak=(\Lambda^0 \otimes \gz) \oplus (\Lambda^1 \otimes \gfrak^1)$ is the even part of $\ufrak.$

Let $\{e_a\}$ denote a basis of $\gz$ and $\{f_{\alpha}\}$ a basis of $\gfrak^1.$
Elements of $\hfrak$ take the form $\sum_a \lambda_a e_a+\sum \mu_{\alpha}f_{\alpha}$
where $\lambda_a\in \Lambda^0$ for each $a$ and $\mu_{\alpha}\in \Lambda^1$ for each $\alpha,$ while elements of $\nfrak$ require the additional condition that the body of $\lambda_a$ be zero for each a.  Thus there is bijection from $\nfrak$ onto the Banach space $\widetilde{\K^{m|n}}$ where 
$\widetilde{\K^{m|n}}$ is the kernel of the body mapping from $\K^{m|n}$ onto $\K^m$
defined by $(\lambda_a,\mu_{\alpha}) \rightarrow (body( \lambda_a)).$ (Recall that 
$\Lambda $ is a Banach algebra and that $\K^{m|n}$ is a Banach space with norm defined by 
$||(\lambda_a,\mu_{\alpha})||=\sum_a ||\lambda_a ||_{\Lambda}+\sum_{\alpha} ||\mu_{\alpha}||_{\Lambda}.$)

Observe that we can choose the bases $\{e_a\}$ and $\{f_{\alpha} \}$ such that 
$||e_a||=1=||f_{\alpha}||$ and if we do this then the bijection from $\nfrak$ onto $\widetilde \K^{m|n}$
referred to above is a norm preserving bijection with a norm preserving inverse.

Also note that 
if $X\in \nfrak, $ then the mapping $ad_X:\nfrak \rightarrow \nfrak$ defined by 
$ad_X(Y)=[X,Y]$ is a $\Lambda^0$ linear even endomorphism from $\nfrak$ to $\nfrak.$
It follows from Lemma ~\ref{L:spectrum} that $\lambda \ I_{\nfrak} - ad_X $ is invertible iff $\lambda \neq 0.$

Consequently, the spectrum of $ad_X$ is zero for each $X\in \nfrak$
and so $\nfrak$ is quasi-nilpotent in the sense defined by Wojtyn'ski \cite{W}.

Now according to Wojtyn'ski it follows that the global Baker-Campbell-Hausdorff (B-C-H) formula holds for $\nfrak.$ Moreover, the usual local Lie group operation 
$\diamond$ defined on $\nfrak$ by $X\diamond Y= \Theta (X,Y)$ is a globally defined operation on $\nfrak$ relative to which $\nfrak$ is a global simply connected Banach Lie group.
In this case the exponential function is a bijection and in fact can be identified with the identity mapping.

One may now follow the proof of Rogers \cite{R} on pages 115--117 and construct a simply 
connected Banach Lie group $H$ whose Lie algebra is $\hfrak.$ Moreover this Banach Lie
group is in fact a super Lie group in the sense that its charts take on values in $\K^{m|n}$
and which are $G^{\infty}$ compatible. We outline this proof to be convincing that it is
valid in the $G^{\infty}$ category.

The group $H$ turns out to be a semi-direct product of  $\widetilde G$ and $N$ where $\widetilde G$ is the unique simply connected finite-dimensional Banach Lie group having the Banach Lie algebra $\gfrak^0$ as its Lie algebra and where $N=\nfrak$ is the quasi-nilpotent Lie group discussed above.

The action of $\widetilde G$ on $N$ is obtained via a sequence of steps described as follows.
First consider the representation $Adj^0$ of $\gz$ on $\gfrak$ defined by $Adj^0(X)(Y)=[X,Y].$
Let $(aut \ \gfrak)_0$ denote the identity component of the group of even automorphisms of $\gfrak$
and define a representation $\pi:\widetilde G\rightarrow (aut \ \gfrak)_0$ by requiring that
$\pi(exp_{\gz}(X))=exp_{End(\gfrak)} (Adj^0(X)),$ for each $X\in \gz.$ Next extend $\pi$ to 
$\pi':\widetilde G\rightarrow aut(\hfrak)$
by the requirement that $\pi'(g)(aX)=a\pi(g)(X)$ for all $a\in \Lambda, X\in \gfrak^{|a|}.$ Now observe that $\pi'(g)(\nfrak)\subseteq \nfrak$ for each $g\in \widetilde G$ and define $\pi''(g)=\pi'(g)|_{\nfrak}.$
Finally, define the desired action $\alpha$ of $\widetilde G$ on $N$ by requiring that, for each $g\in \widetilde G,$
$$\alpha(g)\circ exp_{\nfrak}=exp_{\nfrak} \circ \pi''(g).$$

With this action we define a group operation $\circ$ on $H=\widetilde G\times N$ via 
$$(g_1,n_1)\circ (g_2,n_2)=(g_1g_2, n_1\alpha(g_1)(n_2)).$$

\noindent Thus $H=\widetilde G \rtimes N$ with respect to the action $\alpha.$ Following Rogers, let $V,U,U_1$ be connected coordinate neighborhoods about the identity $e$ in $\widetilde G$ such that $UU\subseteq V, UU_1^{-1}\subseteq U.$
Let $\phi_e$ denote a $C^{\infty}$ chart of $\widetilde G$ defined on $U_1,$ thus $\phi_e(U_1)$ is an open subset of the Lie algebra $\gz.$ For each $g\in \widetilde G$ , define $U_g=U_1g,$ and 
$$\phi_g:U_g\rightarrow \gz  \quad \quad  hg\mapsto \phi_e(h).$$
Recall that there is a bijection $\eta$ from $N=\nfrak$ onto $\widetilde{\K^{m|n}}$ which we may regard as a global chart with values in the Banach space $\widetilde{\K^{m|n}}.$ Then
$\phi_g \times \eta $ may be identified as a chart on the open subset $U_g \times N$ with image
the open subset $\phi_g(U_g) \times \widetilde{\K^{m|n}}$ of $\gz \times \widetilde{\K^{m|n}}$
which we identify with $\K^m \times \widetilde{\K^{m|n}}= \K^{m|n}.$ It follows that there is an atlas
on the semi-direct product with values in the Banach space $\K^{m|n}.$ By Lemma 3.4  it follows that the transition mappings from one chart to another are $C^{\infty}$ mappings which are $\Lambda^0$ right linear mappings. Consequently it follows from Jadczyk  and Pilch  \cite{ JP}, page 380,section 5, that the transition mappings are class $G^{\infty}$ mappings and we have a $G^{\infty}$ atlas. It follows that $H$ is a $G^{\infty}$ super Lie group. We show that its Lie algebra is $\hfrak.$

Let $\{X_i\}_{i=1}^{m+n}$ denote a basis of $\gfrak$ such that 
 $\{X_i\}_{i=1}^m$ denotes a basis of $\gz$ and  $\{X_i\}_{i=m+1}^{m+n}$a basis of $\gfrak^1.$
Relative to this basis, elements of $\hfrak$ take the form 
$\sum_{i=1}^m {\lambda}^i  X_i +\sum_{a=1}^{m+n}\mu^a X_{a+m}$
where $\lambda^i\in \Lambda^0$ for each $i$ and $\mu_a\in \Lambda^1$ for each a. Elements of $\nfrak$ require the additional condition that the body of $\lambda^i$ be zero for each i.

We follow Rogers who shows us how to identify the Lie algebra of the super Lie group H. Since
$\gz$ may be identified with the Lie algebra of $\widetilde G$ it follows from Proposition \ref{P:LocalLie} 
that we can write elements $g_1,g_2\in \widetilde G$ near the identity
as $$g_1=exp(\sum_{i=1}^m tx_1^i X_i), \quad g_2=exp(\sum_{a=1}^m tx_2^i X_i)$$
\noindent  for real $x_1^i,x_2^i$ and small real t.

Recall from Rogers Theorem Part I  we found a Banach basis $X^J_i=\zeta^J\otimes X_i$ of 
$\ufrak$ as a Banach Lie algebra over $\K.$ It follows that elements of $\nfrak$ may be expanded
in this basis with coefficients in $\K$ and consequently we may write elements $n_1,n_2$ near the identity in $N$ as 
$$n_1=exp(\sum_{i=m+1}^{m+n} t x_{1J}^i X_i^J), \quad n_2=exp(\sum_{i=m+1}^{m+n} t x_{2J}^i X_i^J)$$
\noindent for real $x_{1J}^i, x_{2J}^i$ and  small real $t$ and where there is an implied sum over the multi-indices $J$ such that $|J|$ is nonzero and such that $|J|$ is even when $X_i$ is even and $|J|$ is odd when $X_i$ is odd.
The remainder of the proof that $\hfrak$ is the Lie algebra of $H$ follows Roger's argument at the bottom of page 116 and at the top of page 117 word for word with only minor change in notation and where in our context the sums over the multi-indices $J$ (in her notation $\mu$) range over all multi-indices without bound on $|J|.$

Thus we have the following Theorem:

\begin{theorem} Let $\gfrak=\gz \oplus \gfrak^1$ denote a $\Z_2$ graded Banach Lie algebra and let $\ufrak=\Lambda \otimes \gfrak$ denote the corresponding conventional super Lie algebra. Let 
$\hfrak$ denote the even $\Lambda^0$ submodule $(\Lambda^0 \otimes \gz) \oplus (\Lambda^1\otimes \gfrak^1) $of $\ufrak.$ Let $\widetilde G$ denote the unique simply connected finite dimensional Banach Lie group with Lie algebra $\gz$ and let $N$ denote the quasi-nilpotent Lie group having Lie algebra the quasi-nilpotent ideal of the body mapping $\beta$ from $\hfrak$ onto $\gz.$ Then there is an action $\alpha$  of $\widetilde G$ on $N$ which defines a semi-direct product structure on $\widetilde G \times N.$ If 
$H=\widetilde G\rtimes N$ denotes this semi-direct product then there is a $G^{\infty}$ atlas on $H$ relative to which $H$ is a $G^{\infty}$ super Lie group having Lie algebra the super Lie algebra $\hfrak.$

\end{theorem}

\section{Super Riemannian metrics and Super Spin Groups}

Let $\Mcal$ be a $G^{\infty}$ super manifold of dimension $(m,n)$ and $U$ an open subset of $\Mcal.$ We say that $X$ is a \underline {super vector field} on $U$ iff $X$ is a mapping with domain $U$ such that \newline
(1) for each $p\in U, \ X(p)$ is a mapping from $G^{\infty}(U) $ into $\Lambda$ such that $X(fg)=X(f)g(p)+(-1)^{|X| |f|}f(p)X(g)$ for $f,g \in  G^{\infty}(U),$ and \newline
(2) the mapping $X(f)$ defined by $p\rightarrow X(p)(f)$ is a $G^{\infty}$ mapping. \newline

We denote the $\Lambda $ module of all super vector fields on $U$ by $\Xcal(U).$ We follow Rogers \cite{Rsuperman, R} who  introduced these ideas in the manner utilized here.

\begin{definition} Let $\Mcal$ be a $G^{\infty}$ super manifold. We say that $g$ is a super Riemannian metric on $\Mcal$ iff $g$ is a mapping whose value $g(p)=g_p$ at each $p\in \Mcal$ is a mapping from 
$T_p\Mcal \times T_p\Mcal $ into $\Lambda$ which satisfies the following conditions.
For each pair of super vector fields $X,Y \in \Xcal(U)$ defined on an open subset $U$ of $\Mcal,$ the mapping $g(X, Y)$ from $U$ into $\Lambda$ defined by $p\rightarrow g_p(X(p),Y(p))$ is a $G^{\infty}$ mapping on $U$ such that for $Z \in \Xcal$ and $\lambda, \mu \in \Lambda,$

\[  \begin{array}{lcc}

(1) \  g(\lambda X +\mu Y,Z)=\lambda g( X ,Z)+\mu g(Y,Z) \\
(2) \ g(X, Y)=(-1)^{|X| |Y|} g(Y, X) \\
(3) \ g(\cdot,Z)=0  \quad    if \ and \ only \ if  \quad  Z=0.\\
\end{array} \] 

It follows that if either $X$ or $Y$ is even $g(X,Y)=g(Y,X)$ while for odd $X$ and $Y, $ $g(X,Y)=-g(Y,X).$ It is also true that for even $g,$ $g(X,\mu Y)=(-1)^{|\mu| |X|}g(X,Y).$
Here the parity of $g$ is defined by 
$$|g(X,Y)|=|g|+|X|+|Y|.$$

We additionally require that our super Riemannian metric satisfy the condition:

\bigskip
 (4) $g$ is even, thus , for all $X,Y\in \Xcal(U),$
 $$|g(X,Y)|=|X|+|Y|.$$
\end{definition}

\bigskip

\bigskip

\begin{lemma} There exists a $g$-orthogonal pure basis of $T^0_pM $ for each $p\in M.$

\end{lemma}

\begin{proof} First let $\{e_i\}$ denote a basis of $T^0_pM$ and $\{e^1_{\alpha}\}$ a basis of 
$T_p^1M.$ Let $G$ denote the matrix of $g$ relative to this basis. Then $G$ is invertible and 

     \[G= \left( \begin{array}{ccc}
                                           A    &   C \\
                                                  
                                          D     &   B
                                          \end{array} \right)   \]

\noindent where $A=(g(e_i,e_j))$ and $ B=(g(e^1_{\alpha},e^1_{\beta})),$ are invertible matrices with
all their entries from $\Lambda^0.$ Thus $A$ is an invertible symmetric matrix as is also its 
body $\beta(A).$  Now $\beta(A)$ is a symmetric invertible  matrix with its entries from $\R,$ 
and therefore there 
exists an orthogonal matrix $\Ocal$ such that if $\bar e_i=\Ocal_i^je_j$ and 
$\bar A=(g(\bar e_i,\bar e_j)),$ then $\beta(\bar A)$ is an invertible diagonal matrix over $\R.$
We now modify the usual Gram-Schmit orthogonalization process to diagonalize $\bar A$ itself.
Define a new basis $\{f_i\} $ of $T^0_pM$ inductively as follows.
Define $f_1=\bar e_1.$ For each $k$ let
$$f_{k+1}=
\bar e_{k+1}-g(\bar e_{k+1},f_k)g(f_k,f_k)^{-1}f_k-\cdots -g(\bar e_{k+1},f_1)g(f_1,f_1)^{-1}f_1.$$
We assume inductively that $\{ f_l : 1\leq l \leq k \}$ are orthogonal and  that $g(f_l,f_l)$ is invertible
for each $l.$ In order for this to be meaningful we must show that $f_{k+1}$ is $g$-orthogonal to $f_l$ for each
$1\leq l \leq k$ and that $g(f_{k+1},f_{k+1})$ is invertible in  $\Lambda^0.$

First observe that for $1\leq l  \leq k,$
$$g(f_{k+1},f_l)=g(\bar e_{k+1},f_l)-g(\bar e_{k+1},f_l)g(f_l,f_l)^{-1}g(f_l,f_l)=0.$$
We must now show that $g(f_{k+1},f_{k+1})$ is an invertible element of $\Lambda^0.$
Notice that $\bar e_2=f_2  +af_1=f_2+a \bar e_1$ for some $a\in \Lambda^0, 
\bar e_3=f_3+ bf_1+cf_2=f_3+b\bar e_1+c(\bar e_2-a \bar e_1)$ for some $b,c\in \Lambda^0.$
In general an inductive argument shows that 
$\bar e_{k+1}= f_{k+1}+\sum_{i=1}^ka_i \bar e_i$ for some choice of $a_i\in \Lambda^0.$
Since the $f_i$ are orthogonal and since $g(f_{k+1},f_{l})= 0, 1\leq l \leq k,$
$$g(f_{k+1},f_{k+1})=g(f_{k+1},\bar e_{k+1})=
g(\bar e_{k+1}-\sum_{i=1}^k a_i\bar e_i,\bar e_{k+1})$$
$$=g(\bar e_{k+1},\bar e_{k+1})
-\sum_{i=1}^k a_ig(\bar e_{i},\bar e_{k+1}).$$

\noindent Now $\beta(\bar A)$  is diagonal and invertible, thus 
$$\beta(g(f_{k+1},f_{k+1}))=\beta(g(\bar e_{k+1},\bar e_{k+1}))$$
and consequently, the body of $g(f_{k+1},f_{k+1})$ is invertible in $\R.$ Therefore
$g(f_{k+1},f_{k+1})$ is invertible in $\Lambda^0.$ It now follows that $\{f_1,f_2,\cdots ,f_{k+1} \}$
is an orthogonal set of vectors such that $g(f_i,f_i) $ is invertible in $\Lambda^0$ for each $i.$ The lemma
follows by induction.

\end{proof}

\begin{remark} The last lemma clearly holds in a slightly more general context. If we begin
with local vector fields $\{e_i\}$ which form a  basis of $T_q^0M$ at each point $q$ of a neighborhood $U$ of $p$ and if  $\{e^1_{\alpha}\}$ are vector fields on $U$ which form  a basis of  $T_q^1M$ at each point $q$ of $U$ then there exists local vector fields defined at each point of  $U$ which are pure bases of $T_qM$ at each point $q\in U$ such that the even vector fields in the basis are $g$-orthogonal at each point of $U.$

\end{remark}

\begin{lemma} Let $\{\bar e_i \}$ denote a $g$-orthogonal basis of $T^0_pM$ over $\Lambda^0$
and for each $i$ let $d_i=g(\bar e_i,\bar e_i)$ be invertible in $\Lambda^0.$  Let $E=T^0_pM$ and let
$E^{\perp}$ denote its $g$-orthogonal complement in $T_pM.$ Then $T_pM$ is the orthogonal
direct sum of $E$ and $E^{\perp}.$ Moreover if $\{ e_{\alpha} \}$ is a basis of $T^1_pM$ over $\Lambda^0$ then $\{f_{\alpha} \}$ is a basis of $E^{\perp}$ over $\Lambda^0$ where
$$f_{\alpha}=e_{\alpha}-\sum_j g(d_j\bar e_j,e_{\alpha})e_j$$
\noindent for each $\alpha.$ Finally, $f_{\alpha}$ is odd for each $\alpha$ and so $g$ restricted
to $E^{\perp} \times E^{\perp}$ is skew-symmetric.

\end{lemma}

\begin{proof} Let $x\in E^{\perp}$ and write $x$ in terms of the given basis of $T_pM$ over $\Lambda^0,$ $x=\sum_ix_0^i\bar e_i+\sum_{\alpha}x^{\alpha}e_{\alpha},$ then $g(\bar e_i,x)=0$  for all $i$ and since the $\bar e_i$ are orthogonal,
$$0=g(\bar e_i,x)=g(\bar e_i,\sum_ix_0^j \bar e_j)+g(\bar e_i,\sum_{\alpha}x^{\alpha}e^{\alpha})
=x^i_0d_i+\sum_{\alpha}x^{\alpha}g(\bar e_i, e_{\alpha}).$$
It follows that $x_0^i=-d_i^{-1}\sum_{\alpha}x^{\alpha} g(\bar e_i,e_{\alpha})$ and that 
$$x=-\sum_j \sum_{\alpha} x^{\alpha} g(d_j^{-1}\bar e_j,e_{\alpha})\bar e_j+\sum_{\alpha}x^{\alpha}e_{\alpha}=\sum_{\alpha} x^{\alpha} (e_{\alpha}-\sum_j g(d_j^{-1}\bar e_j,e_{\alpha}) \bar e_j).$$
Thus if $f_{\alpha}=e_{\alpha}-\sum_j g(d_j^{-1}\bar e_j,e_{\alpha}) \bar e_j)$ then $x\in E^{\perp}$
if and only if $x=\sum_{\alpha} x^{\alpha}f_{\alpha}.$ Clearly $\{f_{\alpha} \}$ freely generates $E^{\perp}$ since $\sum_{\alpha}x^{\alpha}f_{\alpha}=0$ implies that 
$$\sum_ix_0^i\bar e_i+\sum_{\alpha}x^{\alpha}e_{\alpha}=0$$
where $x_0^j=d_j^{-1}\sum_{\alpha}x^{\alpha} g(\bar e_j,e_{\alpha}).$
Consequently  $x^{\alpha}=0$ for each $\alpha$  since the $\{\bar e_j,e_{\alpha} \}$ freely generate $T_pM.$ Finally observe that $|g(\bar e_j,e_{\alpha})|=|\bar e_j|+|e_{\alpha}|=1.$ It follows that 
$|f_{\alpha}|=1$ is odd and that $g(f_{\alpha}, f_{\beta})=-g( f_{\beta},f_{\alpha}).$

\end{proof}

\begin{remark} This lemma may also be formulated in terms of local vector fields 
defined in a neighborhood of a point in an obvious manner or in terms of a basis
of local sections of the vector bundles $T^0M\rightarrow M, E^{\perp}\rightarrow M.$
We leave these tweaks of the language to the reader.
\end{remark}

\begin{lemma} Let $\{\bar e_i \}$ denote a set of local vector fields defined on a neighborhood $U$ of $p\in M$ which are $g$-orthogonal over $\Lambda^0$ at each point of $U.$
At each point $q\in U$ and for each $i$ let $d_i(q)=g(\bar e_i(q),\bar e_i(q))$ be invertible in $\Lambda^0.$  Let $E_q=T^0_qM$ for $q\in U$  and let
$E^{\perp}_q$ denote its $g$-orthogonal complement in $T_qM.$  Moreover, let $\{f_{\alpha}(q)\}$ denote the basis of $E^{\perp}_q$ at each $q\in U$ defined in the previous lemma. Then there exists a basis of local sections 
$\{\bar f_{\alpha}\} $ of $E^{\perp}$ defined on an open set of $M$ containing $p$ such that  the matrix $\bar B=(g(\bar f_{\alpha},\bar f_{\beta}))$ of 
$\omega \equiv  g| (E^{\perp} \times E^{\perp})$ has the form $\bar B=\Jcal$ where

     \[ \Jcal \equiv \left( \begin{array}{ccccc}
                                           J    &   0 & 0 & \cdots &0 \\
                                                  
                                          0     &   J  & 0 & \cdots & 0 \\
                                          
                                           &  &         \cdots       &   &               \\
                                                  
                                            0 & 0 &0 &  \cdots &    J      
                                          \end{array} \right)   \]   

\noindent and where

     \[ J=\left( \begin{array}{ccc}
                                           0    &   1 \\
                                                  
                                          -1     &   0
                                          \end{array} \right) .  \]   
\end{lemma}

\begin{proof} Let $B$ denote the matrix $(g(f_{\alpha},f_{\beta})),$ then $B$ is skew-symmetric with all of its entries in $\Lambda^0.$ Since $B$ is invertible so is its body and both $\omega$ and its body $\beta(\omega)$ are non-degenerate. Since $\beta(\omega)$ is non-degenerate the dimension $n=dim(E^{\perp}_q)=dim(T^1_qM)$ must be even for each $q.$ Let $n=2l.$ Choose $\hat f_1=f_1$ and choose $\hat f_{l+1}\in E^{\perp}$ such that $\beta(\omega)(\hat f_1,\hat f_{l+1}) \neq 0.$

\bigskip

\noindent The existence of $\hat f_{l+1}$ follows from the fact that $\beta(\omega)$ is non-degenerate.
Let $z=\omega(\hat f_1,\hat f_{l+1}),$ then $z$ is invertible in $\Lambda^0$ since its body is nonzero. Let $\bar f_1 =\hat f_1= f_1,\bar f_{l+1}=z^{-1} \hat f^{l+1}.$ We have
$$\omega(\bar f_1,\bar f_1)=0 \quad 
\quad  \omega(\bar f_{l+1},\bar f_{l+1})=0  \quad \quad \omega(\bar f_1,\bar f_{l+1})=1.$$

\noindent Let 
$$F_1=\{ z_1\bar f_1+ z_2\bar f_{l+1} \ | \  z_1,z_2 \in \Lambda^0 \}  \quad \ and \ 
\quad \ F_2=\{ x\in E^{\perp} \ | \ \omega(x,F_1)=0 \}, $$
thus $F_2$ is the $\omega$-orthogonal complement of $F_1$ in $E^{\perp}.$
Notice that if $y\in F_1 \cap F_2 $ then $y=w_1\bar f_1+ w_2\bar f_{l+1},  w_1, w_2 \in \Lambda^0$
and $\omega(w_1\bar f_1+ w_2\bar f_{l+1}, z_1\bar f_1+ z_2\bar f_{l+1} )=0$ for all $z_1,z_2 \in\Lambda^0.$ Since $\omega $ is skew, it follows that $(w_1z_2-w_2z_1)\omega(\bar f_1,\bar f_{l+1})=0$ for all $z_1,z_2\in \Lambda^0$  and therefore that $w_1z_2=w_2z_1.$ Choosing $z_1=0,z_2=1$  implies $w_1=0.$ Reversing the choice gives $w_2=0.$ Notice that these choices are indeed in $\Lambda^0.$ So $y=0$ and $F_1 \cap F_2 =0.$ We now show that $E^{\perp}=F_1+F_2.$ Let $x\in E^{\perp},$ we show that
 $x-\omega(x,\bar f_{l+1})\bar f_1+\omega(x,\bar f_1)\bar f_{l+1} $ is in $F_2.$ Indeed, for $z_1,z_2 \in \Lambda^0,$
\[
\begin{array}{lcc}

\omega(x-\omega(x,\bar f_{l+1})\bar f_1+ \omega(x,\bar f_1)\bar f_{l+1} ,z_1\bar f_1 +z_2 \bar f_{l+1})  \\
= z_1\omega(x,\bar f_1)+z_2\omega(x,\bar f_{l+1})-\omega(x,\bar f_{l+1})z_2\omega (\bar f_1,\bar f_{l+1}) +\omega(x,\bar f_1)z_1\omega(\bar f_{l+1},\bar f_1) \\
=z_1\omega(x,\bar f_1)+z_2\omega(x,\bar f_{l+1})-\omega(x,\bar f_{l+1})z_2+\omega(x,\bar f_1)(-z_1)=0.  

 \end {array} \nonumber 
 \]
where we have used the facts that $\omega(\bar f_1,\bar f_1)=0=\omega(\bar f_{l+1},\bar f_{l+1})=0.$\newline

It follows that $y \equiv x-\omega(x,\bar f_{l+1})\bar f_1+\omega(x,\bar f_1)\bar f_{l+1} $ is in $F_2$
and that $x=\omega(x,\bar f_{l+1})\bar f_1-\omega(x,\bar f_1)\bar f_{l+1} +y \in F_1+ F_2.$
If $l=1$ the lemma follows. When $l>1$
repeat the process above on $F_2,$ i.e., choose $\bar f_2=\hat f_2$ to be an arbitrary nonzero element  of $F_2$ and choose $\hat f_{l+2}\in F_2$ such that $\beta(\omega)(\hat f_2,\hat f_{l+2})\neq 0. $ This is possible when $l>1$ since if it were the case that $\omega(\hat f_2, F_2)=0$  then it would follow that $\omega(\hat f_2,E^{\perp})=\omega(\hat f_2,F_1+F_2)=0$ and thus that $\hat f_2=0.$ Following the steps above one is able to split $F_2$ as a direct sum of the $\Lambda^0$ sub-module spanned by $\bar f_2, \bar f_{l+2}$ and its $\omega$-orthogonal complement $F_3$ in $F_2.$  At this stage one has
$$\E_{\perp}=<\bar f_1,\bar f_{l+1}> \oplus <\bar f_2,\bar f_{l+2}> \oplus F_3.$$

An inductive argument then establishes the Lemma.

\end{proof}

The following Theorem is an immediate consequence of these Lemmas.

\begin{theorem} Assume that $g$ is a super Riemannian metric on a supermanifold $M.$
If $p\in M$ then there exists a neighborhood $U$of $p$ and  a basis $\{ \bar e_i, \bar f_{\alpha} \}$  of vector
fields on $U$ over $\Lambda^0$
such that the matrix $\Gamma=\Gamma_g$ of the metric $g$ relative to this basis takes the form:

     \[  \Gamma_g=\left ( \begin{array}{ccc}
                                          \eta    &   0 \\
                                                  
                                          0     &   \Jcal
                                          \end{array} \right)   \]  

\noindent where $\eta$ is diagonal with entries invertible elements $d_i$ of $\Lambda^0,$ and where
$\Jcal$ is the standard symplectic matrix

     \[   \Jcal=\left ( \begin{array}{ccccc}
                                           J    &   0 & 0 & \cdots &0 \\
                                                  
                                          0     &   J  & 0 & \cdots & 0 \\
                                          
                                           &  &         \cdots       &   &               \\
                                                  
                                            0 & 0 &0 &  \cdots &    J      
                                          \end{array} \right) .  \]

\end{theorem}

Since the diagonal elements $d_i$ of the matrix $\eta$ are supernumbers and since
our supernumbers are infinitely generated one cannot generally transform $\eta$ to a
form having only $\pm 1$ on the diagonal. Thus generally $\Gamma_g$ is neither an orthosymplectic nor is 
it a generalized orthosymplectic matrix. 

Notice that the matrix  $\Gamma=\Gamma_g$ depends only on the diagonal elements $d_i$ and on the matrix $\Jcal$ and consequently is dependent on the choice of basis $\{e_i\}.$   Following DeWitt \cite{Dewitt} we find conditions relative to
which $\Gamma_g$ is a generalized orthosymplectic matrix. When these conditions are met, DeWitt \cite{Dewitt} calls this matrix the canonical form
of the super metric $g.$

In general the $d_i$ are functions from  $U$ into $\Lambda^0$ since the $\{e_i\}$ are vector fields on $U.$  

Observe that
in case that one is given  a basis of vector fields $\{e_i, f_{\alpha}\}$  defined on all of $M$ or in case one is satisfied to have a metric defined only locally on an open set $U$ one can choose arbitrary supernumbers $\{d_i\}$ and define a metric by requiring that 
$g(e_i,e_i)=d_i$ and that  $\Jcal_{\alpha \beta}=g(f_{\alpha},f_{\beta}).$ In such a case the $d_i$ are not functions and one can construct super metrics whose canonical forms are not generalized orthosymplectic matrices.   In this case the frame bundle reduces to a group defined by 
$\Gamma$ but this group would not be a generalized orthosymplectic group but something new.

So the question arises: what condition or conditions on $g$ will guarantee that $\Gamma$ is a generalized orthosymplectic matrix.

\begin{theorem} Assume that  $\{ \bar e_i, \bar f_{\alpha} \}$ is a basis of  local vector fields of $TM$ over $\Lambda^0$ defined at $p\in M$
with the properties guaranteed by the previous Theorem. If the diagonal elements 
$d_i=g(\bar e_i,\bar e_i)$
satisfy the condition
$$\frac{||s(d_i)||}{|\beta (d_i)|}<1,$$
then there exists a basis $\{ \hat e_i\}$ of local sections of $T^0M\rightarrow M$ defined on an open $U\subseteq M$ containing $p$ such that the canonical matrix $\Gamma_g$ of the metric $g$ relative  to the basis $\{ \hat e_i, \bar f_{\alpha} \}$ takes the form:

     \[  \Gamma_g=\left  ( \begin{array}{ccc}
                                          \eta    &   0 \\
                                                  
                                          0     &   \Jcal
                                          \end{array} \right)   \]

\noindent where $\eta$ is diagonal with entries the elements $\pm 1$ of $\Lambda^0$ and where
$\Jcal$ is the standard symplectic matrix above. Moreover one can always choose a basis $\{\tilde e_i\}$ of $T^0_pM$ such that $||\tilde e_i||=1$
and if such a basis is chosen the condition above is satisfied if and only if 
$$||s(\tilde d_i)||<\frac{1}{2}.$$

\end{theorem}

\begin{proof}  We show that there exists a supernumber $\lambda \in \Lambda^0$  such that 
$\hat e_i \equiv \lambda \bar e_i$  satisfies the condition:
$$g(\hat e_i,\hat e_i)= \frac{\beta(d_i)}{|\beta(d_i)|}$$
where $d_i=g(\bar e_i,\bar e_i).$ Such a supernumber  clearly must satisfy the condition 
$$\lambda^2(\beta(d_i)+s(d_i))=\lambda^2 d_i=\frac{\beta(d_i)}{|\beta(d_i)|}.$$
Thus we must have $\lambda^2(1+\frac{s(d_i)}{\beta(d_i)})=\frac{1}{|\beta(d_i)|}$
and consequently that 
$$\lambda=\frac{1}{\sqrt {|\beta(d_i)|}} \  \big (1+\frac{s(d_i)}{\beta(d_i)}\big)^{-\frac{1}{2}}.$$
The last equation must be understood in terms of the Binomial Theorem for fractional powers.
 The series expansion of $(1+\mu)^{-\frac{1}{2}}$ converges for supernumbers $\mu$ provided $||\mu||<1.$ Thus $\lambda $ is well-defined if and only if $\mu=\frac{s(d_i)}{\beta(d_i)}$
 satisfies the condition $||\mu||=\frac{||s(d_i)||}{|\beta (d_i)|}<1.$
 
 Next observe that if $\tilde e_i=|| d_i ||^{-\frac{1}{2}}\bar e_i,$ then $||g(\tilde e_i,\tilde e_i)||=1.$ 
Consequently, it is no loss of generality to assume that $||d_i||=1.$ If we do so, then
$1=||d_i||=|\beta(d_i)|+||s(d_i)||$ and
$$\frac{||s(d_i)||}{|\beta (d_i)|}<1 \Longleftrightarrow ||s(d_i)||<|\beta(d_i)| 
\Longleftrightarrow ||s(d_i)||<1-||s(d_i)||$$$$ \Longleftrightarrow 2||s(d_i)||<1
\Longleftrightarrow ||s(d_i)||<\frac{1}{2}.$$

\end{proof}

\begin{definition} Let $g$ be a super Riemannian metric on the supermanifold $M.$ We say $L$ is a local isometry of $g$ iff there is an open subset $U$ of $M$ such that 
for each $q\in U,  L_q$ is an even endomorphism of $T_qM$ such that
$g(L_q(x),L_q(y))=g(x,y), \ x,y\in T_qM.$ 
\end{definition}

Let $L$ define such a local isometry on $U$ and let $\{v_{\alpha}\}$ be a pure basis of local sections of  
$TM\rightarrow M$ defined on $U$
such that the matrix of $g$ relative to this basis is
 
     \[  \Gamma=\Gamma_g= \left ( \begin{array}{ccc}
                                           \eta    &   0 \\
                                                  
                                          0     &   \Jcal
                                          \end{array} \right)   \]  

\noindent  where $\eta$ is a diagonal matrix with entries $\eta_{\alpha \alpha}=g(v_{\alpha},v_{\alpha}), 1 \leq \alpha \leq m.$ If $N$ is an even matrix and 
 
     \[ N= \left( \begin{array}{ccc}
                                            A    &   C \\
                                                  
                                          D     &    B
                                          \end{array} \right),   \]   
                                          
\noindent  then $N$ is the matrix of a local isometry $L$ relative to the basis $\{v_{\alpha}\}$ if and only if                                         
$$N^{ST} \Gamma N= \Gamma$$
where $N^{ST}$ is the super transpose of $N$ and is defined by

     \[ N^{ST}= \left( \begin{array}{ccc}
                                            A^T    &   -D^T \\
                                                  
                                          C^T     &    B^T
                                          \end{array} \right).  \]   
   
\noindent In case $\eta$ has diagonal entries which are ordinary numbers  one can arrange that the nonzero entries of $\eta$ are $\pm 1$ and we can further arrange the entries so that
the positive entries precede the negative entries on the diagonal. In the case that the matrix $\eta$ of the 
"diagonalization"  of  the matrix of $g$ has ordinary numbers as entries  we say that  $g$ is {\bf body reducible}.

Using the notation of the last definition, it is straightforward to show that $N$ is the matrix of a 
local isometry $L$ defined on $U$  iff

\[  \begin{array}{lcc}

(1) \  A^T\eta A-D^T\Jcal D=\eta  \\
(2) \ C^T \eta C+ B^T\Jcal B= \Jcal  \\
(3) \  A^T\eta C -D^T \Jcal B=0 .\\
\end{array} \] 

We intend to find the simply connected Banach Lie group which covers the group $\Gcal_U$ of local isometries of $g$ defined on $U$ and along the way we will show that in case $g$ is body reducible it is a super Lie group.

To do the latter we plan to use our generalization of Roger's Theorem. Thus we first find the 
Lie algebra $Lie( \Gcal_U)$ of the local isometry group $\Gcal_U$ of $g$ and show that it is a conventional super Lie algebra.
We will then use the Roger's result to find the unique simply connected Banach Lie group $\widetilde \Gcal_U$ which has $Lie(\Gcal_U)$ as its Lie algebra. It turns out that since $Lie(\Gcal_U)$ is a conventional super Lie algebra it will follow that $\widetilde \Gcal_U$ is a super Lie group. Consequently, the covering group of
the group of local isometries of a super Riemannian metic is a super Lie group and so has the right to be called the super spin group of the super metric $g.$

First we find the Lie algebra of the Banach Lie group $\Gcal_U.$ To do this
let $\lambda \rightarrow N(\lambda)$ denote a curve through the identity of the local isometry group $\Gcal_U.$
Equations (1),(2),(3) above hold for $A(\lambda), B(\lambda) ,C(\lambda), D(\lambda).$ Differentiating
these equations and setting $\lambda=0$ yields

\[  \begin{array}{lcc}

(1) \  a^T\eta +\eta a=0  \\
(2) \ b^T\Jcal + \Jcal  b=0 \\
(3) \ \eta c -d^T \Jcal =0 \\
\end{array}. \] 

\noindent since $A(0)=Id, B(0)=Id, C(0)=0,D(0)=0.$ Here, of course, $a=A'(0),b=B'(0),c=C'(0)$ and 
$d=D'(0).$ It is not difficult to check that the Lie algebra $Lie(\Gcal_U)$
of $\Gcal_U$ is the set of all matrices

     \[ \left( \begin{array}{ccc}
                                            a    &  c= \eta^{-1} d^T\Jcal \\
                                                  
                                           d     &    b
                                          \end{array} \right)  \]   
                                          
\noindent such that   $a^T\eta =-\eta a ,    b^T\Jcal =- \Jcal  b$ and where all the entries of the matrices $a$ and
$b$ are even supernumbers while all the entries of the matrices $c$ and $d$ are odd supernumbers.
This Lie algebra is the Lie algebra of a Banach Lie group so it is also the Lie algebra of its covering group $\widetilde \Gcal_U. $

\underline {In the special case that $g$ is body reducible}, we can say more. In this case we can 
apply our generalization of Roger's result in order to find a group
which has $Lie(\Gcal_U)$ as its Lie algebra and which is simply connected.

Let    $\gfrak=\gz \oplus \gfrak^1$ denote the $\Z_2$ graded Lie algebra over $\R$ of dimension 
$(m,n)$  defined as follows. Let $\gz$ denote the set of all matrices                            
       
     \[ \left( \begin{array}{ccc}
                                            a    &   0 \\
                                                  
                                           0    &    b
                                          \end{array} \right) \]   
                                     
 \noindent where $a^T \eta =-\eta a $ and $b^T\Jcal =-\Jcal b,$ but where all the entries of both $a$ and $b$ are
  ordinary numbers. Let $\gfrak^1$ denote the set of all matrices

     \[ \left( \begin{array}{ccc}
                                            0   &   c \\
                                                  
                                           d     &    0
                                          \end{array} \right) \]   
                                          
\noindent such that $\eta c= d^T \Jcal $ where all the entries of $c$ and $d$ are ordinary numbers.
 Now it is easy to show that $\frak g=\gz\oplus \gfrak^1$ is a $\Z_2$ graded Lie algebra over $\R$ of dimension $(m,n)$ and that $\Lambda \otimes \gfrak$ is a conventional super Lie algebra.
 
 To apply the generalization of Rogers Theorem proved above, let  $\ufrak=\Lambda \otimes \gfrak$
 and let $\hfrak=
  (\Lambda^0 \otimes \gz) \oplus (\Lambda^1 \otimes \gfrak^1) $ denote the even Lie subalgebra of $\ufrak.$   It is easy to see that $\hfrak=Lie( \Gcal_U).$  Indeed,  $\hfrak$ is precisely the set of all matrices $\ell$
  such that 
  
     \[ \ell=\left( \begin{array}{ccc}
                                            a    &   c \\
                                                  
                                           d     &    b
                                          \end{array} \right)  \]   
                    
  \noindent  where $\ell^{ST}\Gamma=-\Gamma \ell$ and where the entries of $a$ and $b$ are even supernumbers and the entries of $c$ and $d$ are odd supernumbers. 
  
 It follows that $\hfrak$ is precisely $Lie(\Gcal_U)$ since $\ell^{ST}\Gamma=-\Gamma \ell$ 
  iff  
  $$a^T\eta=- \eta a,  \quad b^T \Jcal=-\Jcal b, \quad \eta c=d^T \Jcal $$
 and these are precisely the conditions required in order that $\ell$ belong to $Lie(\Gcal_U).$
 Observe that this argument only holds when $\eta$ has entries ordinary numbers.

The following theorem is an immediate consequence of our remarks above.

\begin{theorem} Assume that the canonical form of the super metric $g$ is body reducible and let $\tilde G$ denote the unique simply connected Lie group of the Lie algebra
$\gz=\beta(Lie (\Gcal_U)).$ Then there exists an infinite dimensional quasi-nilpotent Banach Lie group $N$ 
and an action of $\tilde G $ on $N$ which defines a semi-direct product structure on $\tilde G\times N$ 
such that \newline
(1) $\tilde \Gcal_U =\tilde G \rtimes N$ and \newline
(2) $\tilde G \rtimes N $ is a super Lie group whose super Lie algebra is the conventional
Lie algebra $\Lambda \otimes \gfrak$ where $\gfrak=\gz \oplus \gfrak^1$ is a $\Z_2$ graded
Lie algebra with $\gz=\beta(Lie(\Gcal_U)) $ and $\gfrak^1$ is the vector space of matrices over $\R$ defined above. 

\end{theorem}

\begin{definition} If $g$ is a super Riemannian metric on a super manifold $M$ then the simply
connected covering group $\tilde \Gcal_U$ of the group $\Gcal_U$ of local isometries of $g$ defined on $U$ is called a local
spin group of $g.$
\end{definition}

At this point all we can say in general regarding the canonical form of an arbitrary  super metric is that
its matrix may be written in the form  

     \[ \left( \begin{array}{ccc}
                                          \eta    &   0 \\
                                                  
                                          0     &    \Jcal
                                          \end{array} \right)   \]   

\noindent where $\eta=diag(d_1,d_2,\cdots, d_m).$ We have found conditions relative to which $d_i=\pm1$
for each $i$ but a general result is still pending.

 In case $m=p+q$ and $d_i=1 ,1\leq i \leq p,
d_i=-1, p+1\leq i \leq p+q$ we see that the local isometry group is a generalized orthosymplectic group.
To find the corresponding spin group, recall that this group is the set of all even matrices 

     \[N= \left( \begin{array}{ccc}
                                           A    &   C \\
                                                  
                                          D     &   B
                                          \end{array} \right)   \]

\noindent such that $N^{ST} \Gamma N= \Gamma.$ The Lie algebra of this group is the set of all even matrices 

  \[ \ell=\left( \begin{array}{ccc}
                                            a    &   c \\
                                                  
                                           d     &    b
                                          \end{array} \right)  \]   
                    
  \noindent  where $\ell^{ST}\Gamma=-\Gamma \ell.$ The body mapping sends this Lie algebra onto the set
  of matrices

  \[ \beta(\ell)=\left( \begin{array}{ccc}
                                            \beta(a)    &   0 \\
                                                  
                                           0     &    \beta(b)
                                          \end{array} \right)  \]   
  
\noindent where $\beta(a)^T\eta=-\eta \beta(a)$ and  $ \beta(b)^T \Jcal=-\Jcal \beta(b).$ 

We denote this Lie algebra by $\beta(Lie \  \Gcal_U)$ and observe that its Lie group is $SO(p,q) \times Sp(n)$
where $Sp(n)$ is the symplectic group and  we consider only the nontrivial case with $n$  even. Now this Lie
group is not simply connected and its simply connected covering group $\widetilde G$ which is required by our last Theorem above 
is the group $\widetilde G=Spin(p,q) \times Mp(n)$ where $Mp(n)$ is the metaplectic group.
Now the quasi-nilpotent group $N$ required by our Theorem is obtained as follows. As in our generalization of
Rogers Theorem let $\frak h=(\Lambda^0\otimes \frak g^0)\oplus (\Lambda^1\otimes \frak g^1)$ where 
$\frak g^0=\beta(Lie \ \Gcal_U)$ and $\frak g^1$ is the set of matrices with real entries of the form
 
     \[ \left( \begin{array}{ccc}
                                            0   &   c \\
                                                  
                                           d     &    0
                                          \end{array} \right) \]   
                                          
\noindent where $c=\eta d^T \Jcal. $ Thus $\frak h$ is the set of even matrices of the form

     \[ \ell=\left( \begin{array}{ccc}
                                            a    &   c \\
                                                  
                                           d     &    b
                                          \end{array} \right)  \] 
  
   \noindent and the Lie algebra $\frak n=ker \beta $ is the set of matrices in $\frak h$ of the form
   
     \[ \ell=\left( \begin{array}{ccc}
                                            s(a)    &   c \\
                                                  
                                           d     &    s(b)
                                          \end{array} \right).  \]   
    The group $N$ is simply $\frak n$ with a group operation $\diamond$ defined on it by                                       
                                          
   \[ exp \big[\left ( \begin{array}{cccccccccccc}
                                            s(a_1)    &   c_1 \\
                                                  
                                           d_1     &    s(b_1)
                                          \end{array} \right) 
 \diamond \left ( \begin{array}{ccc}
                                            s(a_2)    &   c_2 \\
                                                  
                                           d_2     &    s(b_2)
                                          \end{array} \right)  \big ] 
=exp\left ( \begin{array}{cccccc}
                                            s(a_1)    &   c_1 \\
                                                  
                                           d_1     &    s(b_1)
                                          \end{array} \right) 
 exp\left ( \begin{array}{ccc}
                                            s(a_2)    &   c_2 \\
                                                  
                                           d_2     &    s(b_2)
                                          \end{array} \right) .  \]

In this case the spin group of the super metric $g$ is a semi-direct product of the simply connected group 
$Spin(p,q) \times Mp(n)$ and the group $N=\frak n$ with operation $\diamond$ defined above.. In the special case that $m=p,q=0$ the local isometry group is the orthosymplectic group and the spin group of the super metric $g$ is a semi-direct product of $Spin(p) \times Mp(n)$ and $N.$  In a case of physical interest whereby the super metric is locally super Lorentzian, i.e.,  where 
$p=1,q=3$ the spin group is a semi-direct product of $Sl(2,\C)\times Mp(n)$ and a quasi-nilpotent group $N$ defined as above.

Given a super metric $g$ we may define a principal bundle $\Fcal_g$ over $M$ whose
fiber over $p\in M$ is the set of all pure bases of $T_pM$ where $(X_i)_{i=1}^m$ is
a basis of $T^0M$ and $(X_i)_{i=m+1}^{m+n}$ is a basis of its $g$-orthogonal complement $E^{\perp}$
defined in one of the Lemmas above. In the case that the canonical form $\Gamma_g$ of $g$ is a generalized orthosymplectic matrix the bundle of frames $\Fcal_g$ clearly reduces to a principal fiber bundle $(\Ocal\Scal)_g$with the generalized orthosymplectic group $Osp(p,q,n)$ as its structure group. 

 As usual the converse is also true, i.e., any reduction of the bundle $\Fcal_g$ to the group $Osp(p,q,n)$  defines a super metric whose canonical matrix is generalized orthosymplectic.

The general case is more complex. In case one can fix invertible supernumbers $\{d_i\}_{i=1}^m$ and obtain
a global canonical form $\Gamma_g$ defined in terms of these constant supernumbers as 
suggested above then one has a reduction of the bundle $\Fcal_g$ to a bundle  $\Ocal\Scal(d_i)$ whose structure group
is the group  $O(d_i,n)=\{ N \ | \ N^{ST} \Gamma N= \Gamma \}$ where

     \[ \Gamma= \left( \begin{array}{ccc}
                                          \eta    &   0 \\
                                                  
                                          0     &    \Jcal
                                          \end{array} \right)   \]   

\noindent with $\eta=diag(d_1,d_2,\cdots, d_m).$ In this case one obtains a generalization of the
generalized orthosymplectic group and its corresponding bundle.

The spin bundle $\widetilde \Scal_g$ of any one of these bundles $\Ocal \Scal(d_i)$  is defined, as usual, to be a bundle having the covering
group $\widetilde{ O(d_i,n)}$ of  $O(d_i,n)$as structure group and which covers the bundle $\Ocal \Scal(d_i)$
in the usual fashion. The issue as to when such bundles exist is not addressed here but certainly  is of interest.
Observe that quasi-nilpotent groups are contractible so the factor $N$ of $\widetilde{ O(d_i,n)}$ offers no 
obstruction to the existence of $\widetilde \Scal_g.$

In this context one defines spinor fields, as usual, to be sections of vector bundles associated to 
representations of the structure group of the spin bundle $\widetilde \Scal_g.$  In the case that the structure group is the covering  group of the group $Osp(p,q,n)$ one requires  representations of the group 
$(Spin(p,q) \times Mp(n))\rtimes \frak n.$

Representations of $Spin(p,q)$ induce representations of \newline
$(Spin(p,q) \times Mp(n))\rtimes \frak n$
as do  also representations of $Mp(n).$ So ordinary spinor fields and symplectic spinors become  local  spinor fields arising from  the super metric. 
In general any representation of any closed subgroup of $(Spin(p,q) \times Mp(n))\rtimes \frak n$ induces a representation
of $(Spin(p,q) \times Mp(n))\rtimes \frak n$ and consequently gives rise to local super spinor fields in the sense described above.
In particular representations of $\frak n$ will induce local super spinor fields in the sense used here. The representation theory of
nilpotent Lie groups is a well-developed field of study, but the author knows of no research dealing with the representations 
of quasi-nilpotent Banach Lie groups.

It is beyond the scope of this paper  (and the author) to classify all representations of the group  
$(Spin(p,q) \times Mp(n))\rtimes \frak n.$ 
It  is of interest to find all  irreducible representations of $(Spin(p,q) \times Mp(n))\rtimes \frak n.$ The problem of finding all irreducible representations is difficult. Mackey has proved Theorems characterizing all irreducible unitary representations of groups which are semi-direct products of groups of the form $L\rtimes A$ where $L$ is any  locally compact group and $A$ is a locally compact abelian group, but the present author knows of no generalization which computes irreducible representations of such a semi-direct product when $A$ is a general nilpotent Lie group. The case when $A$ is quasi-nilpotent is even more remote since such groups are not locally compact.

 \end{document}